\documentclass[12pt,twoside,a4paper]{amsart}
\usepackage{amssymb,verbatim,enumerate,ifthen,hyperref}
\usepackage[utf8]{inputenc}
\usepackage[T1]{fontenc}
\usepackage[mathscr]{eucal}

\newtheorem*{theorem}{Theorem}
\newtheorem*{lemma}{Lemma}
\newtheorem*{corollary}{Corollary}

\def\eq#1{{\rm(\ref{E#1})}}
\def\Eq#1#2{\ifthenelse{\equal{#1}{*}}
  {\begin{equation*}\begin{aligned}#2\end{aligned}\end{equation*}}
  {\begin{equation}\begin{aligned}\label{E#1}#2\end{aligned}\end{equation}}}

\newcommand{\R}{\mathbb{R}}
\newcommand{\N}{\mathbb{N}}

\newcommand{\A}{\mathscr{A}}
\newcommand{\G}{\mathscr{G}}
\renewcommand{\H}{\mathscr{H}}

\newcommand{\AG}{\mathscr{A}\otimes\mathscr{G}}
\def\maple#1{{\begin{verbatim}> {#1}\end{verbatim}}}

\begin{document}
\begin{flushright}
Comput. Math. Appl. \textbf{58} (2009) 334--340. \\
\href{http://dx.doi.org/10.1016/j.camwa.2009.03.107}{doi: 10.1016/j.camwa.2009.03.107} \\[1cm]
\end{flushright}
\vspace{5mm}

\title[Invariance equation for two variable Gini means]
{Computer aided solution of the invariance equation for two-variable Gini means}
\author[Sz. Baják]{Szabolcs Baják}
\author[Zs. Páles]{Zsolt Páles}

\email{\{bajaksz,pales\}@science.unideb.hu}
\address{Institute of Mathematics, University of Debrecen, 
4010 Debrecen, Pf 12, Hungary}
\date\today
\thanks{\textit{Corresponding author:} Zsolt Páles, 
Institute of Mathematics, University of Debrecen, 
4010 Debrecen, Pf 12, Hungary}
\thanks{This research was supported by the Hungarian
Scientific Research Fund (OTKA) Grants K-62316, NK-68040.}

\begin{abstract}
Our aim is to solve the so-called invariance equation in the class of two-variable 
Gini means $\{G_{p,q}:p,q\in\R\}$, i.e., to find necessary and sufficient 
conditions on the 6 parameters $a,b,c,d,p,q$ such that the identity
\[
   G_{p,q}\big(G_{a,b}(x,y),G_{c,d}(x,y)\big)=G_{p,q}(x,y) \qquad (x,y \in \R_+)
\]
be valid. We recall that, for $p\neq q$, the Gini mean $G_{p,q}$ is defined by
\[
  G_{p,q}(x,y):=\left(\dfrac{x^p+y^p}{x^q+y^q}\right)^{\frac1{p-q}}\qquad (x,y \in \R_+).
\]
The proof uses the computer algebra system Maple V Release 9 to 
compute a Taylor expansion up to 12th order, which enables
us to describe all the cases of the equality.
\end{abstract}

\maketitle

\section{Introduction}

Let $\R_+$ denote the set of positive real numbers
throughout this paper. A two-variable continuous
function $M:\R_+^2\to \R_+$ is called a \emph{mean} on $\R_+$ if  
\Eq{M}{
  \min(x,y)\leq M(x,y)\leq\max(x,y)\qquad(x,y\in \R_+)
}
holds. If both inequalities in \eq{M} are strict whenever $x\neq y$,
then $M$ is called a \emph{strict mean} on $\R_+$.

Given two means $M,N:\R_+^2\to \R_+$ and $x,y\in\R_+$, 
the iteration sequence
\Eq{*}{
    x_1 & :=x,\qquad &
    y_1 & :=y, \\
  x_{n+1} & :=M(x_n,y_n),\qquad &
  y_{n+1} & :=N(x_n,y_n)\qquad\quad (n\in\N)
}
is said to be the \emph{Gauss-iteration} determined by the pair
$(M,N)$ with the initial values $(x,y)\in \R_+^2$. It is well-known
(cf.\ \cite{BB87}, \cite{DarPal02c}) that if $M$ and $N$ are strict means then the
sequences $(x_n)$ and $(y_n)$ are convergent and have equal limits
$M\otimes N(x,y)$ which is a strict mean of the values $x$ and $y$.
The mean $M\otimes N$ defined by this procedure is called the
\emph{Gauss composition} of $M$ and $N$.

A key result characterizing the Gauss composition of means is
the following statement: Given two strict means $M,N:\R_+^2\to \R_+$,
their Gauss composition $K=M\otimes N$ is the unique strict mean
solution $K$ of the functional equation
\Eq{MNK}{
  K\big(M(x,y),N(x,y)\big)=K(x,y)\qquad(x,y\in \R_+)
}
which is called the \emph{invariance equation}.

The simplest example when the invariance equation holds is the
well-known identity
\Eq{*}{
  \sqrt{xy}=\sqrt{\dfrac{x+y}2\cdot\dfrac{2xy}{x+y}}\qquad(x,y\in\R_+),
}
that is,
\Eq{*}{
  \G(x,y)=\G\big(\A(x,y),\H(x,y)\big)\qquad(x,y\in\R_+),
}
where $\A,\G$, and $\H$ stand for the two-variable arithmetic,
geometric, and harmonic means, respectively. Another less trivial
invariance equation is the identity 
\Eq{*}{
  \AG(x,y)=\AG\big(\A(x,y),\G(x,y)\big)\qquad(x,y\in\R_+),
}
where $\AG$ denotes Gauss' \textit{arithmetic-geometric mean} defined by
\Eq{*}{
  \AG(x,y)=\bigg( \frac2{\pi}\int_0^{\frac{\pi}2}
   \frac{dt}{\sqrt{x^2\cos^2t+y^2\sin^2 t}} \bigg)^{-1}\qquad
   (x,y\in\R_+). 
}

The invariance equation in more general classes of means has recently been 
studied extensively by many authors in various papers. The invariance of the arithmetic 
mean $\A$ with respect to two quasi-arithmetic means was first investigated by Matkowski
\cite{Mat99a} under twice continuous differentiability assumptions concerning
the generating functions of the quasi-arithmetic means. These regularity
assumptions were weakened step-by-step by Dar\'oczy, Maksa, and Páles in the
papers \cite{DarMakPal00}, \cite{DarPal01a}, and finally this problem was completely 
solved assuming only continuity of the unknown functions involved \cite{DarPal02c}.
The invariance equation involving three weighted quasi-arithmetic means was
studied by Burai \cite{Bur06}, \cite{Bur07} and Jarczyk--Mat\-kowski
\cite{JarMat06}, Jarczyk \cite{Jar07}. The final answer (where no additional
regularity assumptions are required) has been obtained in \cite{Jar07}.
In a recent paper, we have studied the invariance of the arithmetic mean with 
respect to two so-called generalized quasi-arithmetic mean under four times
continuous differentiability assumptions \cite{BajPal09a}.
The invariance of the arithmetic mean with respect to Lagrangian means was the
subject of investigation of the paper \cite{Mat05} by Matkowski.
The invariance of the arithmetic, geometric, and harmonic means with respect
to the so-called Beckenbach--Gini means was studied by Matkowski in \cite{Mat02a}.
Pairs of Stolarsky means for which the geometric mean is invariant were
determined by B{\l}asi{\'n}ska-Lesk--G{\l}a\-zowska--Matkowski
\cite{BlaGlaMat03}. The invariance of the arithmetic mean with respect to
further means was studied by G{\l}azowska--Jarczyk--Matkowski
\cite{GlaJarMat02} and Domsta--Matkowski \cite{DomMat06}.

An important class of two variable homogeneous means are the so-called
Gini means (cf.\ Gini \cite{Gin38}). Given two parameters $p,q\in\R$,
the two-variable mean $G_{p,q}:\R_+^2\to \R_+$ is defined by the 
following formula
\Eq{*}{
  G_{p,q}(x,y)=
  \begin{cases}
  	\bigskip
  	\left(\dfrac{x^p+y^p}{x^q+y^q}\right)^{\frac1{p-q}} 
           & \mbox{ for } p\neq q,\\
  	\exp \left({\dfrac{x^p \ln x + y^p \ln y}{x^p+y^p}}\right) 
           & \mbox{ for } p=q,
  \end{cases}
}
for $x,y \in\R_+$.

The class of Gini means is a generalization of the class of power means,
since taking $q=0$, we immediately get the power (or H\"{o}lder) mean of 
exponent $p$. 

The aim of this paper is to solve the invariance equation in the class
of Gini means, i.e., to solve \eq{MNK} when each of the means $M,N$,
and $K$ is a Gini mean. More precisely, we want to describe the set of all 
6-tuples $(a,b,c,d,p,q)$ such that the identity
\Eq{IE}{
  G_{p,q}\big(G_{a,b}(x,y),G_{c,d}(x,y)\big)=G_{p,q}(x,y)
    \qquad(x,y\in\R_+)
}
holds. The main result of this paper is contained in the following theorem.

\begin{theorem}
Let $a,b,c,d,p,q\in\R$. Then the invariance equation \eq{IE} is satisfied if and only if 
one of the following possibilities hold:
\begin{enumerate}[(i)]
\item $a+b=c+d=p+q=0$, i.e., all the three means are equal to the geometric mean,
\item $\{a,b\}=\{c,d\}=\{p,q\}$, i.e., all the three means are equal to each other,
\item $\{a,b\}=\{-c,-d\}$ and $p+q=0$, i.e., $G_{p,q}$ is the geometric mean and 
      $G_{a,b}=G_{-c,-d}$,
\item there exist $u,v\in\R$ such that $\{a,b\}=\{u+v,v\}$, $\{c,d\}=\{u-v,-v\}$, and 
      $\{p,q\}=\{u,0\}$ (in this case, $G_{p,q}$ is a power mean),
\item there exists $w\in\R$ such that $\{a,b\}=\{3w,w\}$, $c+d=0$, and $\{p,q\}=\{2w,0\}$
      (in this case, $G_{p,q}$ is a power mean and $G_{c,d}$ is the geometric mean),
\item there exists $w\in\R$ such that $a+b=0$, $\{c,d\}=\{3w,w\}$, and $\{p,q\}=\{2w,0\}$
      (in this case, $G_{p,q}$ is a power mean and $G_{a,b}$ is the geometric mean).
\end{enumerate}
\end{theorem}

As an obvious consequence, we obtain the following solution for the so-called 
Matkowski--Sut\^o equation, i.e., when $G_{p,q}$ is equal to the arithmetic mean in \eq{IE}.

\begin{corollary}
Let $a,b,c,d\in\R$. Then the Matkowski--Sut\^o equation 
\Eq{*}{
  G_{a,b}(x,y)+G_{c,d}(x,y)=x+y\qquad(x,y\in\R_+)
}
is satisfied if and only if one of the following possibilities hold:
\begin{enumerate}[(i)]
\item $\{a,b\}=\{c,d\}=\{1,0\}$, i.e., the two means are equal to the arithmetic mean,
\item there exist $v\in\R$ such that $\{a,b\}=\{1+v,v\}$, $\{c,d\}=\{1-v,-v\}$,
\item $\{a,b\}=\{\frac32,\frac12\}$ and $c+d=0$ (in this case, $G_{c,d}$ is the geometric mean),
\item $a+b=0$ and $\{c,d\}=\{\frac32,\frac12\}$ (in this case, $G_{a,b}$ is the geometric mean).
\end{enumerate}
\end{corollary}

In view of the homogeneity of Gini means, identity \eq{IE} is equivalent to the equation
\Eq{IE+}{
  \frac{G_{p,q}\big(G_{a,b}(\operatorname{e}^{x},\operatorname{e}^{-x}),
	G_{c,d}(\operatorname{e}^{x},\operatorname{e}^{-x})\big)}
        {G_{p,q}(\operatorname{e}^{x},\operatorname{e}^{-x})}=1 \qquad (x\in\R).
}
The main and simple idea of the proof of our Theorem is to compute the Taylor expansion of 
the function $F:\R\to\R$ defined by
\Eq{F}{
  F(x):=\frac{G_{p,q}\big(G_{a,b}(\operatorname{e}^{x},\operatorname{e}^{-x}),
	G_{c,d}(\operatorname{e}^{x},\operatorname{e}^{-x})\big)}
        {G_{p,q}(\operatorname{e}^{x},\operatorname{e}^{-x})}
}
at the point $x=0$. 
Using the symmetry of Gini means, it also follows that $F$ is an even function. 
Therefore all coefficients $C_k$ of odd order in the Taylor expansion of $F$ 
are equal to zero automatically. The validity of \eq{IE+} yields that all of the even 
order coefficients also vanish. To derive the necessity of the conditions in 
our Theorem, we will only use the equalities $C_2=\cdots=C_{12}=0$,
which produce 6 equations for the 6 unknown parameters $a,b,c,d,p,q$. The sufficiency
of the conditions so obtained will be proved by a simple argument which, implicitely, shows
that the equalities $C_2=\cdots=C_{12}=0$ yield $C_{2k}=0$ for $k\geq 7$.
Unfortunately, these 6 equations are so complicated that their solution requires the 
power of the Maple computer algebra package.
This will be done in the next section, where we also show the sequence of Maple commands
that were performed during our computation. If the interested reader executes all these 
commands in a Maple worksheet, then all the computations can be repeated and checked again.
During this computation, we analyze each of the Taylor coefficients separately and 
use the information so obtained at the subsequent steps.

\section{The proof of Theorem}

First we recall the characterization of the equality of two variable Gini means.

\begin{lemma}(Cf.\ \cite{Pal88c}) Let $a,b,c,d\in\R$. Then the identity
\Eq{G}{
  G_{a,b}(x,y)=G_{c,d}(x,y)\qquad(x,y\in\R_+)
}
holds if and only if one of the folowing possibilities is valid:
\begin{enumerate}[(i)]
\item $a+b=c+d=0$ and, in this case, the two means are equal to the 
geometric mean,
\item $\{a,b\}=\{c,d\}$.
\end{enumerate}
\end{lemma}

\begin{proof}[Proof of the Theorem]
Assume that \eq{IE} holds. Then the function $F$ defined by \eq{F} 
is identically $1$ on $\R$. Thus, for the $k$th-order
Taylor coefficient $C_k$ defined below, we have:
\Eq{C_k}{
   C_k:=\frac{F^{(k)}(0)}{k!}=0 \qquad (k\in\N).
}
Since $F$ is even, thus $C_k=0$ for all odd $k\in\N$.

In view of the symmetry of the Gini means in the parameters, we may assume that
$a\geq b$, $c\geq d$, and $p\geq q$ in the sequel.

For the calculations in Maple, define the Gini mean $G_{p,q}(x,y)=G(p,q,x,y)$ and 
the function $F$ given by \eq{F} by performing the following commands:
\begin{verbatim}
> G:=(p,q,x,y)->((x^p+y^p)/(x^q+y^q))^(1/(p-q));
  F:=x->(G(p,q,G(a,b,exp(x),exp(-x)),G(c,d,exp(x),exp(-x))))
     /(G(p,q,exp(x),exp(-x)));
\end{verbatim}
This yields the following output:
\Eq{*}{
   G:&= ( {p,q,x,y} )\mapsto \left( \frac{x^p+y^p}{x^q+y^q} \right)^{\frac{1}{p-q}}\\
   F:&= x\mapsto \frac{G\big(p,q,G(a,b,\operatorname{e}^{x},\operatorname{e}^{-x}),  
G(c,d,\operatorname{e}^{x},\operatorname{e}^{-x})\big)}{G(p,q,\operatorname{e}^{x},\operatorname{e}^
{-x})}
}
In the first step, we evaluate the 2nd-order Taylor coefficient $C_2$ by executing 
\begin{verbatim}
> C[2]:=simplify(coeftayl(F(x),x=0,2));
\end{verbatim}
This results
\Eq{*}{
  C_2:=\frac{1}{4}a+\frac{1}{4}b+\frac{1}{4}c+\frac{1}{4}d-\frac{1}2p-\frac{1}2q
}
Therefore, by $C_2=0$, we obtain our first necessary condition: $(a+b+c+d)/4=(p+q)/2$. 
If one tries to compute $C_4,C_6,\dots$, then the expressions obtained are so
complicated that it is hard to get further information. 
In order to simplify the evaluation of the higher-order Taylor coefficients, 
we introduce the following notations
\Eq{*}{
  w:=&\frac{a+b+c+d}{4}=\frac{p+q}2, \\
  v:=&\frac{a+b-(c+d)}{4},\\
  t:=&\left( \frac{p-q}2\right)^2,\\
  r:=&\frac{(a-b)^2+(c-d)^2}{8},\\
  s:=&\frac{(a-b)^2-(c-d)^2}{8}.
}
(In the definition of $w$ we utilized the condition $C_2=0$.)
Then we can express the parameters $a,b,c,d,p,q$ in the following form:
\begin{verbatim}
> a:=w+v+sqrt(r+s); b:=w+v-sqrt(r+s); 
  c:=w-v+sqrt(r-s); d:=w-v-sqrt(r-s); 
  p:=w+sqrt(t); q:=w-sqrt(t);
\end{verbatim}
\Eq{*}{
  a&:= w+v+\sqrt {r+s}\\  b&:= w+v-\sqrt {r+s}\\
  c&:= w-v+\sqrt {r-s}\\  d&:= w-v-\sqrt {r-s}\\
}
\Eq{*}{
  p&:= w+\sqrt {t}\\  q&:= w-\sqrt {t}
}
Now we evaluate the 4th order Taylor coefficient by inputting: 
\begin{verbatim}
> C[4]:=simplify(coeftayl(F(x),x=0,4));
\end{verbatim}
Then we obtain
\[
  C_4:=\frac{1}{3}tw-\frac{1}{3}vs-\frac{1}{3}wr
\]
The condition $C_4=0$ yields that $wt=wr+vs$. 

If $w=0$, then $p+q=0$ and hence $G_{p,q}$ is equal to the 
geometric mean. Therefore, the invariance equation \eq{IE} can be rewritten as
\Eq{*}{
  G_{a,b}(x,y)G_{c,d}(x,y)=xy \qquad(x,y\in\R_+).
}
This results
\Eq{*}{
  G_{a,b}(x,y)=\frac{1}{G_{c,d}(1/x,1/y)}=G_{-c,-d}(x,y)\qquad(x,y\in\R_+).
}
Using the Lemma again, this identity yields that either 
$a+b=c+d=0$ or $\{a,b\}=\{-c,-d\}$ must hold. Together with $p+q=0$, these 
equations show that either condition (i) or condition (iii) of our theorem 
must be satisfied. Conversely, if conditions (i) or (iii) hold then an easy
computation yields that \eq{IE} is satisfied.

In the rest of the proof, we assume that $w$ is not zero. Then, we can express 
$t$ in terms of $w,v,r,s$:
\begin{verbatim}
> t:=r+v*s/w;
\end{verbatim}
\Eq{t}{
  t:= r+\frac{vs}w
}
Next, we evaluate the 6th order Taylor coefficient:
\begin{verbatim}
> C[6]:=simplify(coeftayl(F(x),x=0,6));
\end{verbatim}
\Eq{*}{
  C_6:=\frac{-2(-3w^2s^2+3v^2s^2-15w^2rv^2-5w^3sv-10v^3sw+15w^4v^2)}{45w}
}

If $v=0$, then the condition $C_6=0$ simplifies to $w^2s^2=0$, whence $s=0$ follows. 
Therefore \eq{t} yields $t=r$ and we obtain that $a=c=p$ and $b=d=q$ which means that
condition (ii) of our theorem must be fulfilled. In this case, it is obvious
that \eq{IE} is satisfied.

In the rest of the proof, we assume that $v$ is also not zero. 
Observe that the 6th order coefficient $C_6$ does not involve higher-order 
powers of $r$. Therefore, the equation $C_6=0$ can be solved for $r$. Temporarily, 
we denote this solution by $R$:
\begin{verbatim}
> R:= (15*w^4*v^2-3*w^2*s^2+3*v^2*s^2-5*w^3*v*s-10*w*v^3*s)/(15*w^2*v^2);
\end{verbatim}
\Eq{R}{
  R:= \frac{15w^4v^2-3w^2s^2+3v^2s^2-5w^3vs-10wv^3s}{15w^2v^2}
}
Finally, we evaluate the 13th order Taylor polynomial of $F(x)$ at $x=0$ (the 
Maple output is suppressed by putting \, : \, instead of \, ; \, to the end of the Maple command, 
for the sake of brevity), then we extract the 8th, 10th and 12th order Taylor 
coefficients, replace $r$ by $R$ and factorize the expressions so obtained by inputting:
\begin{verbatim}
> T:=simplify(taylor(F(x),x=0,13)):
  for i from 8 to 12 by 2 
      do C[i]:=simplify(subs(r=R,simplify(coeff(T,x,i))),factor) od;
\end{verbatim}
\Eq{*}{
C_8&:=\frac {(v-w)(v+w)s}{70875w^3v^2} 
\big(2100w^{3}v^{5}-3850w^{2}v^{4}s+4200w^{5}v^{3}-255wv^{3}s^{2}+153v^{2}s^{3}\\
&\hspace{1cm}-9245w^{4}v^{2}s-7395w^{3}vs^{2}-153w^{2}s^{3}\big)
}
\Eq{*}{
C_{10}&:=\frac {2(v-w)(v+w)s}{1063125w^5v^4}
\big(28500w^5v^9-59675w^4v^8s+34470w^3v^7s^2+20100w^7v^7\\
&\hspace{1cm}-299575w^6v^6s-
     4260w^2v^6s^3-73200w^5v^5s^2-930wv^5s^4+66600w^9v^5\\
&\hspace{1cm}+4805w^4v^4s^3-
		 286500w^8v^4s+279v^4s^5-169020w^7v^3s^2-16740w^3v^3s^4\\
&\hspace{1cm}-558w^2v^2s^5+
      45955w^6v^2s^3+17670w^5vs^4+279w^4s^5\big)
}
\Eq{*}{
C_{12}&:=\frac {2(v-w)(v+w)s}{2631234375w^7v^6} 
\big( -3272692500w^{10}v^8s+22181100w^3v^9s^4-54365475w^6v^8s^3\\
&\hspace{4mm}-25317375w^8v^6s^3+335826w^4v^2s^7-559710wv^7s^6-215221875w^6v^{12}s\\
&\hspace{4mm}+22875570w^6v^4s^5+16977870w^5v^3s^6-7649370w^3v^5s^6-1246797750w^8v^{10}s\\
&\hspace{4mm}-34684335w^8v^2s^5-777170000w^{11}v^5s^2-159926550w^7v^5s^4-335826w^2v^4s^7\\
&\hspace{4mm}+641072375w^{10}v^4s^3+270963000w^5v^{11}s^2-1046615000w^9v^7s^2+11659365w^4v^6s^5\\
&\hspace{4mm}-133190250w^5v^7s^4-1967022000w^{12}v^6s+177650700w^9v^3s^4-8768790w^7vs^6\\
&\hspace{4mm}+385915750w^7v^9s^2+149400w^2v^8s^5-98002025w^4v^{10}s^3+76725000w^7v^{13}\\
&\hspace{4mm}-57172500w^9v^{11}-478665000w^{11}v^9+188100000w^{13}v^7-111942w^6s^7+111942v^6s^7\big)
}
(In the entire Maple computation, this is the only step which requires a considerable 
processing time. On a DualCore 2.6 GHZ processor, the elapsed time was less than 1 minute.)

$C_8, C_{10}$ and $C_{12}$ are obviously zero if $s(v-w)(v+w)=0$. 
Thus, we have to consider the three subcases: $s=0$, $v=w$, and $v=-w$.

In the case $s=0$, \eq{t} and \eq{R} imply that $t=r=w^2$. Then we get that
\Eq{*}{
  \{a,b\}=\{2w+v,v\},\qquad \{c,d\}=\{2w-v,-v\},\qquad \{p,q\}=\{2w,0\},
}
i.e., condition (iv) holds with $u:=2w$. Conversely, if condition (iv) holds and $u\neq0$,
then we have
\Eq{uv}{
  G_{p,q}&\big(G_{a,b}(x,y),G_{c,d}(x,y)\big)
    =G_{u,0}\big(G_{u+v,v}(x,y),G_{u-v,-v}(x,y)\big)\\
   &=\left(\frac{x^{u+v}+y^{u+v}}{2(x^{v}+y^{v})}
          +\frac{x^{u-v}+y^{u-v}}{2(x^{-v}+y^{-v})}\right)^{\frac{1}{u}}
    =\left(\frac{x^{u+v}+y^{u+v}}{2(x^{v}+y^{v})}
          +\frac{x^{u}y^v+y^{u}x^v}{2(x^{v}+y^{v})}\right)^{\frac{1}{u}}\\
   &=\left(\frac{x^{u}+y^{u}}{2}\right)^{\frac{1}{u}}=G_{u,0}(x,y)=G_{p,q}(x,y).
}
Thus \eq{IE} is satisfied if $u\neq0$. If $u=0$, then the parameters also fulfil 
condition (iii), hence \eq{IE} holds in this case, too.

If $v=w$, then, by \eq{t} and \eq{R}, $r=w^2-s$ and $t=r+s=w^2$, respectively.
Hence, 
\Eq{*}{
  \{a,b\}=\{3w,w\},\qquad c+d=0,\qquad \{p,q\}=\{2w,0\},
} 
i.e., condition (v) holds. Conversely, if condition (v) holds and $w\neq0$,
then, using the identity \eq{uv} with $u:=2w$, $v:=w$, we have
\Eq{*}{
  G_{p,q}\big(G_{a,b}(x,y),G_{c,d}(x,y)\big)
    &=G_{2w,0}\big(G_{2w+w,w}(x,y),G_{2w-w,-w}(x,y)\big)\\
    &=G_{2w,0}(x,y)=G_{p,q}(x,y),
}
which shows that \eq{IE} is fulfilled. If $w=0$, then all the three means are geometric means,
and hence \eq{IE} holds trivially.

The last case when $v=-w$ holds, similarly to the case $v=w$, implies that condition 
(vi) is valid. If condition (vi) holds, then \eq{IE} can also be verified.

In the rest of the proof, we can assume that $s(v+w)(v-w)$ is not zero.
The Taylor coefficients $C_8$, $C_{10}$, and $C_{12}$ are of the form
\Eq{*}{
  C_8=\frac {(v-w)(v+w)s}{70875w^3v^2}P_8,\quad 
  C_{10}=\frac {2(v-w)(v+w)s}{1063125w^5v^4}P_{10},\quad
  C_{12}=\frac {2(v-w)(v+w)s}{2631234375w^7v^6}P_{12},
}
where $P_8$, $P_{10}$, and $P_{12}$ are polynomials of the variables $v,w,s$.
They can be obtained by the following Maple commands (whose output is suppressed):
\begin{verbatim}
>  P[8]:=op(5,C[8]): P[10]:=op(5,C[10]): P[12]:=op(5,C[12]):
\end{verbatim}
The equalities $C_8=C_{10}=C_{12}=0$ and $s(v+w)(v-w)\neq0$ imply that 
$P_8=P_{10}=P_{12}=0$. In what follows, we show that there is no solution 
$v,w,s$ to this system of equations.

The variable $s$ is a common root of the polynomials $P_8$ and $P_{10}$. 
Therefore the resultant $R_{8,10}$ of these two polynomials (with respect to $s$) is zero:
\begin{verbatim}
> R[8,10]:=factor(resultant(P[8],P[10],s));
\end{verbatim}
\Eq{*}{
R_{{8,10}}:=& 136687500{w}^{15}{v}^{15} (v-w)^2(v+w)^2\\
   &\big( 1178440166794705680\,{v}^{18}-34849488132334981400\,{w}^{2}{v}^{16}\\
   & +27095657773476976150\,{w}^{4}{v}^{14}+2157163953185024831539\,{v}^{12}{w}^{6}\\
   & +19335728720363587723895\,{w}^{8}{v}^{10}+77098340762854904758838\,{w}^{10}{v}^{8}\\
   & +135541716064734053550290\,{w}^{12}{v}^{6}+52974528518488497499557\,{v}^{4}{w}^{14}\\
   & +2100034048587009260985\,{w}^{16}{v}^{2}+44498612407766474466\,{w}^{18}\big)  
}
Since $vw(v-w)(v+w)\neq0$ holds, we get that $v$ and $w$ are solutions of a 
homogeneous two variable polynomial equation of degree 18. Writing $w$ in the form
\begin{verbatim}
> w:=z*v;
\end{verbatim}
\Eq{*}{
  w:= zv
}
we get that $z$ is a root of a 18th degree polynomial $P_{8,10}$ (more usefully, 
$z^2$ is a root of a 9th degree polynomial), where:
\begin{verbatim}
> P[8,10]:=simplify(op(4,R[8,10])/v^18);
\end{verbatim}
\Eq{*}{
P_{8,10}:=&1178440166794705680-34849488132334981400{z}^{2}\\
   &+27095657773476976150{z}^{4}+2157163953185024831539{z}^{6}\\
   &+19335728720363587723895{z}^{8}+77098340762854904758838{z}^{10}\\
   &+135541716064734053550290{z}^{12}+52974528518488497499557{z}^{14}\\
   &+2100034048587009260985{z}^{16}+44498612407766474466{z}^{18}.
}
The variable $s$ is also a common root of the polynomials $P_8$ and $P_{12}$. 
Therefore the resultant $R_{8,12}$ of these two polynomials (with respect to $s$)
is again zero:
\begin{verbatim}
> R[8,12]:=factor(resultant(P[8],P[12],s)):
\end{verbatim}
Since $vw(v-w)(v+w)\neq0$ holds, we now get that $v$ and $w$ are solutions of a 
homogeneous two variable polynomial of degree 26, whence we get that $z$ is a 
root of the 26th degree polynomial $P_{8,12}$, where:
\begin{verbatim}
> P[8,12]:=simplify(op(4,R[8,12])/v^26);
\end{verbatim}
{\small\Eq{*}{
P&[8, 12] := 8196063700595383871701091232-179090512353635410423157248720z^2\\
&-2262574745604112043731392907114z^4+11198535065282946302316347517923z^6\\
&+369075355861065090753396085824722z^8+3321203212966063219800014204539694z^{10}\\
&+17018221168597358591328346358640128z^{12}+55161742271395394206883716537690208z^{14}\\
&+113024609788553283598449985201081964z^{16}+136472191224999845881431378284988722z^{18}\\
&+83840233563357841801204648333566258z^{20}+19391722782753178903737004919064981z^{22}\\
&+1234978033803167388960240130106010z^{24}+95711050739605210548400442203992z^{26}
}}
Now computing the resultant of the two polynomials $P_{8,10}$ and $P_{8,12}$ by
\begin{verbatim}
> Q:=resultant(P[8,10],P[8,12],z):
\end{verbatim}
it follows that $Q$ is a (huge) nonzero number, hence $P_{8,10}$ and $P_{8,12}$ cannot 
have a common root. This proves that $R_{8,10}$ and $R_{8,12}$ can be simultaneously zero
if and only if $vw(v-w)(v+w)=0$ holds. Hence $C_8=C_{10}=C_{12}=0$ can only hold
also in this case.
\end{proof}


\begin{thebibliography}{10}

\bibitem{BB87}
J.~M. Borwein and P.~B. Borwein, \emph{Pi and the {A}{G}{M}}, John Wiley \&
  Sons Inc., New York, 1987, A study in analytic number theory and
  computational complexity. \MR{89a:11134}

\bibitem{DarPal02c}
Z.~Dar{\'o}czy and Zs. P{á}les, \emph{Gauss-composition of means and the solution 
  of the {M}atkowski--{S}ut\^o problem}, Publ. Math. Debrecen \textbf{61} (2002),
  no.~1-2, 157--218. \MR{2003j:39061}

\bibitem{Mat99a}
J.~Matkowski, \emph{Invariant and complementary quasi-arithmetic means},
  Aequationes Math. \textbf{57} (1999), no.~1, 87--107. \MR{2000g:39025}

\bibitem{DarMakPal00}
Z.~Dar{\'o}czy, Gy. Maksa, and Zs. P{á}les, \emph{{E}xtension theorems for
  the {M}atkowski-{S}ut{\^o} problem}, Demonstratio Math. \textbf{33} (2000),
  no.~3, 547--556. \MR{2002a:39027}

\bibitem{DarPal01a}
Z.~Dar{\'o}czy and Zs. P{á}les, \emph{{O}n means that are both
  quasi-arithmetic and conjugate arithmetic}, Acta Math. Hungar. \textbf{90}
  (2001), no.~4, 271--282. \MR{2003g:26034}

\bibitem{Bur06}
P.~Burai, \emph{Extension theorem for a functional equation}, J. Appl. Anal.
  \textbf{12} (2006), no.~2, 293--299.

\bibitem{Bur07}
P.~Burai, \emph{A {M}atkowski--{S}ut\^{o} type equation}, Publ. Math. Debrecen
  \textbf{70} (2007), no.~1-2, 233--247.

\bibitem{JarMat06}
J.~Jarczyk and J.~Matkowski, \emph{Invariance in the class of weighted
  quasi-arithmetic means}, Ann. Polon. Math. \textbf{88} (2006), no.~1, 39--51.

\bibitem{Jar07}
J.~Jarczyk, \emph{Invariance of weighted quasi-arithmetic means with continuous
  generators}, Publ. Math. Debrecen \textbf{71} (2007), no.~3-4, 279--294.

\bibitem{BajPal09a}
Sz. Baj{á}k and Zs. P{á}les, \emph{Invariance equation for generalized quasi-arithmetic means},
  Aequationes Math. (2009).

\bibitem{Mat05}
J.~Matkowski, \emph{Lagrangian mean-type mappings for which the arithmetic mean is
  invariant}, J. Math. Anal. Appl. \textbf{309} (2005), no.~1, 15--24.
  \MR{2006c:26051}

\bibitem{Mat02a}
J.~Matkowski, \emph{{O}n invariant generalized {B}eckenbach-{G}ini means},
  Functional Equations --- Results and Advances (Z.~Dar{\'o}czy and Zs.
  P{á}les, eds.), Advances in Mathematics, vol.~3, Kluwer Acad. Publ.,
  Dordrecht, 2002, pp.~219--230. \MR{2003j:39063}

\bibitem{BlaGlaMat03}
J.~B{\l}asi{\'n}ska-Lesk, D.~G{\l}azowska, and J.~Matkowski, \emph{An
  invariance of the geometric mean with respect to {S}tolarsky mean-type
  mappings}, Results Math. \textbf{43} (2003), no.~1-2, 42--55.
  \MR{2004a:39042}

\bibitem{GlaJarMat02}
D.~G{\l}azowska, W.~Jarczyk, and J.~Matkowski, \emph{Arithmetic mean as a
  linear combination of two quasi-arithmetic means}, Publ. Math. Debrecen
  \textbf{61} (2002), no.~3-4, 455--467. \MR{2003h:26045}

\bibitem{DomMat06}
J.~Domsta and J.~Matkowski, \emph{Invariance of the arithmetic mean with
  respect to special mean-type mappings}, Aequationes Math. \textbf{71} (2006),
  no.~1-2, 70--85. \MR{2007a:26051}

\bibitem{Gin38}
C.~Gini, \emph{{D}i una formula compressiva delle medie}, Metron \textbf{13}
  (1938), 3--22.

\bibitem{Pal88c}
Zs. P{á}les, \emph{Inequalities for sums of powers}, J. Math. Anal. Appl.
  \textbf{131} (1988), no.~1, 265--270. \MR{89f:26024}

\end{thebibliography}
\def\MR#1{}

\end{document}